\documentclass[final]{amsart}
\usepackage{microtype}
\usepackage{amsmath}
\usepackage{amssymb}
\usepackage{amsthm}
\usepackage{amscd}
\usepackage{mathtools}
\usepackage{marginnote}
\usepackage{todonotes}
\usepackage{bbm}
\usepackage{bm}
\usepackage[notcite,notref]{showkeys}
\usepackage[utf8]{inputenc}
\usepackage{tikz-cd}
\usepackage{cite}
\usepackage[colorlinks=true,linkcolor=black,citecolor=black,urlcolor=blue]{hyperref} 
\usepackage{cleveref}
\newtheorem{theorem}{Theorem}[section]
\newtheorem*{theorem*}{Theorem}

\newtheorem*{conjecture*}{Conjecture}

\newtheorem*{question*}{Question}

\newtheorem*{guess*}{Guess}

\newtheorem*{problem*}{Problem}

\newtheorem{lemma}[theorem]{Lemma}
\newtheorem*{lemma*}{Lemma}

\newtheorem*{exercise*}{Exercise}
\newtheorem{proposition}[theorem]{Proposition}
\newtheorem*{proposition*}{Proposition}

\newtheorem{corollary}[theorem]{Corollary}
\newtheorem*{corollary*}{Corollary}

\theoremstyle{definition}

\newtheorem{definition}[theorem]{Definition}
\newtheorem*{definition*}{Definition}
\newtheorem{remark}[theorem]{Remark}

\newtheorem*{example*}{Example}
\newtheorem*{examples*}{Examples}

\newcommand{\twomat}[4]{\begin{pmatrix} #1 & #2 \\ #3 & #4 \end{pmatrix}}

\usepackage{mathtools}

\newcommand{\FF}{\mathbb{F}}

\newcommand{\QQ}{\mathbb{Q}}

\newcommand{\ZZ}{\mathbb{Z}}

\newcommand{\Cc}{\mathcal{C}}

\newcommand{\Oc}{\mathcal{O}}

\newcommand{\mf}{\mathfrak{m}}

\newcommand{\rarrow}{\rightarrow}

\newcommand{\onto}{\twoheadrightarrow}
\newcommand{\into}{\hookrightarrow}

\newcommand{\normal}{\lhd}

\newcommand{\Hom}{\operatorname{Hom}}

\newcommand{\Tor}{\operatorname{Tor}}

\newcommand{\ind}{\operatorname{ind}}

\newcommand{\cok}{\operatorname{cok}}
\newcommand{\im}{\operatorname{im}}

\usepackage[utf8]{inputenc}
\usepackage{cleveref}
\usepackage{verbatim}
\usepackage{todonotes}
\newcommand{\AG}{A\left\langle G \right\rangle}
\newcommand{\AH}{A\left\langle H \right\rangle}

\newcommand{\Kfin}{\text{$K$-fin}}

\title[Finitely presented smooth mod $p$ representations of $GL_2(F)$.]{The category of finitely presented smooth mod
  $p$ representations of $GL_2(F)$.}

\author{Jack Shotton}

\begin{document}

\maketitle
\begin{abstract}
  Let $F$ be a finite extension of $\mathbb{Q}_p$.  We prove that the category of finitely presented smooth $Z$-finite
  representations of $GL_2(F)$ over a finite extension of $\FF_p$ is an abelian subcategory of the category of all
  smooth representations.  The proof uses amalgamated products of completed group rings.
\end{abstract}
\section{Introduction}
\label{sec:intro}

Let $\FF$ be a finite field of characteristic $p$.  If $G$ is a locally profinite topological group, let $\Cc_{\FF}(G)$
be the category of smooth representations of $G$ over $\FF$.  Throughout this paper, if $K$ is an open subgroup of such
a group $G$ then $\ind_K^G$ denotes induction with compact support modulo $K$.

\begin{definition} \label{def:fgfp} Let $V$ be a smooth $\FF$-representation of a locally profinite group $G$.  Then $V$
  is:
  \begin{enumerate}
  \item \textbf{finitely generated} if for some compact open subgroup $K$ of $G$ there is a surjection of
    $\FF[G]$-modules
    \[\ind_{K}^G W \onto V\] for a smooth finite-dimensional $\FF$-representation $W$ of $K$;
  \item \textbf{finitely presented} if for some compact open subgroups $K_1$, $K_2$ of $G$ there is an exact sequence
    \[\ind_{K_1}^G W_1 \rarrow \ind_{K_2}^G W_2  \rarrow V \rarrow 0\]
    for $W_1$ and $W_2$ smooth finite-dimensional $\FF$-representations of $K_1$ and $K_2$ respectively.
  \end{enumerate}
\end{definition}

Let $F$ be a finite extension of $\QQ_p$.  The purpose of this article is to prove:

\begin{theorem} \label{thm:main} The category of finitely presented smooth $\FF$-representations of $SL_2(F)$ is an
  abelian subcategory of $\Cc_{\FF}(SL_2(F))$.

  The same holds for the category of finitely presented smooth $Z$-finite representations of $GL_2(F)$.
\end{theorem}

This is Theorem~\ref{thm:sl2} and Corollary~\ref{cor:gl2} below.  In fact, we prove the same result with $F$ replaced by
any finite dimensional division algebra over $\QQ_p$.

The theorem is equivalent to the statement that the kernel\footnote{and the cokernel, but this is automatic} of any map
between finitely presented smooth representations is itself finitely presented.  If $\Cc_{\FF}(SL_2(F))$ were the category of
modules over a ring $R$, this would be the statement that $R$ is a coherent ring.  Indeed, we will prove the theorem by
considering smooth $\FF$-representations as modules over the amalgamated product
\[\FF[[K]] *_{\FF[[I]]} \FF[[K']],\] where $K = SL_2(\Oc_F)$, $K' = \twomat{1}{0}{0}{\pi}K\twomat{1}{0}{0}{\pi^{-1}}$ for
$\pi$ a uniformising element of $D$, and $I = K \cap K'$.  Then a result of \r{A}berg \cite{MR680365} shows that, under
certain conditions, an amalgamated product of coherent rings over a noetherian ring is itself coherent.
\emph{Throughout, unless otherwise stated, by `module', `noetherian' or `coherent' we
  mean `left module', `left noetherian' or `left coherent'.}

Finitely presented representations of $GL_2(F)$ were previous studied by Hu \cite{MR2862375}, Vigneras
\cite{MR2906918}, and Schraen \cite{MR3365778}.\footnote{The definition of `finitely presented' in these articles is
  slightly different to ours, and automatically entails $Z$-finiteness.}  In particular, \cite{MR2906918} Theorem~6
shows that a smooth \emph{admissible} finitely presented representation of $GL_2(F)$ has finite length, and that all of
its subquotients are also admissible and finitely presented.  On the other hand, the main result of \cite{MR3365778}
says that, if $F$ is a quadratic extension of $\QQ_p$, then an irreducible supersingular representation of $GL_2(F)$
admitting a central character is never finitely presented.

We are motivated by the construction (see \cite{MR3529394}) of a `patched module' $M_\infty$ that has an action of
$G = GL_n(F)$ and, hopefully, interpolates the hypothetical $p$-adic Langlands correspondence.  It is (only?) possible
to directly obtain information about $M_\infty$ by considering $\Hom_{GL_n(F)}(\ind_K^G(W), M_\infty^\vee)$ for locally
algebraic representations of $K = GL_n(\Oc_F)$ on finitely generated $\ZZ_p$-modules $W$.  This leads us to consider the
category of finitely presented representations of $G$; it also motivates us to prove a version of Theorem~\ref{thm:main} with
coefficients.

I do not know whether Theorem~\ref{thm:main} holds when $G = GL_n(F)$ (or any $p$-adic Lie group).  The method of this
paper does not apply, because $G$ is not (up to centre) an amalgam of two compact open subgroups.  I am not sure whether
Theorem~\ref{thm:main} holds when $F$ has positive characteristic; the method of this paper fails because $GL_2(\Oc_F)$
is not $p$-adic analytic and its completed group ring is not noetherian.  I thank Billy Woods for a helpful discussion
about this case.

I am grateful to Matthew Emerton for asking me the question that this paper answers, and for several helpful and
motivational conversations.  I also thank Julien Hauseux and Stefano Morra for comments and corrections.  I am indebted
to the anonymous referee for suggesting that I relate the amalgamated product of rings considered here to the ring
$\Lambda(G)$ considered in \cite{MR3682662}, which greatly clarified and simplified the arguments of this paper.

\section{Finitely presented representations.}

For the rest of this article, let $\FF$ be a finite field of characteristic $p$. Let $A$ be a complete local noetherian
$W(\FF)$-algebra with maximal ideal $\mf$ and residue field $\FF$.  Let $G$ be a locally profinite group.
Recall (\cite{MR2667882} definition~2.2.5) that a \emph{smooth} $A$-representation of $G$ is a representation of
$G$ on a torsion $A$-module $V$ such that every $v \in V$ is fixed by a compact open subgroup of $G$.  

\begin{definition}
  If $K$ is a profinite group, then a \emph{finite rank} $A$-representation of $K$ is a representation of $K$ on a
  finitely generated $A$-module $M$ such that, for every $n \geq 0$, $M/\mf^nM$ is a smooth representation of $K$.
\end{definition}

Strictly speaking, we should call these finite rank continuous $A$-representations of $K$.

\begin{definition} A representation of $G$ on an $A$-module $V$ is \emph{$K$-finite} if for some (equivalently, any)
  compact open subgroup $K \subset G$, and for every $v \in V$, the $A[K]$-module generated by $v$ is a finite rank
  $A$-representation of $K$.
\end{definition}

We let $\Cc^{\Kfin}_A(G)$ be the category of all $K$-finite $A$-representations of $G$, with morphisms being morphisms
of $A[G]$-modules.
Note that a representation of $G$ on a \emph{torsion} $A$-module $V$ is smooth if and only if it is $K$-finite.

In the introduction (Definition~\ref{def:fgfp}) we gave the definitions of `finitely generated' and `finitely presented'
smooth $\FF$-representations of $G$.  We now extend those to $K$-finite $A$-representations.  First, note that if $M$ is
a finite rank $A$-representation of a compact open subgroup $K\subset G$, then $\ind_K^G M$ is certainly
$K$-finite.
\begin{definition} Let $V$ be a $K$-finite $A$-representation of $G$.  Then $V$ is:
  \begin{enumerate}
  \item \emph{finitely generated} if there is a compact open subgroup $K \subset G$, a finite rank $A$-representation
    $W$ of $K$, and a surjection of $A[G]$-modules
    \[\ind_K^G(W) \onto V;\] 
  \item \emph{finitely presented} if for some compact open subgroups $K_1$, $K_2$ of $G$ there is an exact sequence of
    $A[G]$-modules
    \[\ind_{K_1}^G W_1 \rarrow \ind_{K_2}^G W_2  \rarrow V \rarrow 0\]
    for $W_1$ and $W_2$ finite rank $A$-representations of $K_1$ and $K_2$.
  \end{enumerate}
\end{definition}

We start by establish some straightforward properties of finitely presented $K$-finite representations.  Many of the proofs
follow those of the properties of finitely presented modules over a ring given in \cite[Tag 0519]{stacks-project}.

\begin{lemma}\label{lem:fg-grp-ring} A $K$-finite $A$-representation $V$ of $G$ is finitely generated if and only if it is
  finitely generated as an $A[G]$-module.
\end{lemma}
\begin{proof} For any $W$ and $K$, $\ind_K^G W$ is generated (as an $A[G]$-module) by the finitely generated
  $A$-submodule of functions supported on $K$.  The `only if' direction follows.

  For the `if' direction, let $V$ be a $K$-finite representation generated by $v_1, \ldots, v_n$ as an $A[G]$-module.
  Choose a compact open subgroup $K$ and let $W$ be the finite rank $A$-representation of $K$ generated by
  $v_1, \ldots, v_n$.  Then $V$ is a quotient of $\ind_K^G W$.
\end{proof}

\begin{remark} \label{rmk:not-true} It is not true that a finitely presented $K$-finite $A$-representation of $G$ will
  be finitely presented as an $A[G]$-module; this is already false for the $\FF$-representation $\ind_K^G \FF$, as long
  as $K$ is not finitely generated.  This is the main technical problem that we have to overcome in the next section.
\end{remark}

\begin{lemma} \label{lem:fg-basic} Suppose that $0 \rarrow V_1 \rarrow V_2 \rarrow V_3\rarrow 0$ is a short exact
  sequence of $K$-finite $A$-representations of $G$.

  If $V_1$ and $V_3$ are finitely generated, so is $V_2$.
\end{lemma}
\begin{proof} This is immediate from Lemma~\ref{lem:fg-grp-ring} and the fact that an
  extension of finitely generated modules over $A[G]$ is finitely generated.
\end{proof}

  \begin{lemma}\label{lem:fp-basic} Suppose that $0 \rarrow V_1 \rarrow V_2 \rarrow V_3\rarrow 0$ is a short exact
    sequence of $K$-finite $A$-representations of $G$.
  \begin{enumerate}
  \item If $V_2$ is finitely presented and $V_1$ is finitely generated, then $V_3$ is finitely presented.
  \item If $V_3$ is finitely presented and $V_2$ is finitely generated, then $V_1$ is finitely generated.
  \item If $V_1$ and $V_3$ are finitely presented, so is $V_2$.
  \end{enumerate}
\end{lemma}
\begin{proof} We use $K$ and $L$, $M$, $N$ to denote a suitably chosen compact open subgroup of $G$ and finite rank
  $A$-representations of $K$.
  \begin{enumerate}
  \item Choose a presentation $\ind_{K}^G N \xrightarrow{\alpha} \ind_{K}^G M \rarrow V_2 \rarrow 0$ and choose
    $v_1, \ldots, v_r$ generating the image of $V_1$ in $V_2$ as an $A[G]$-module.  For each $i$, let $\tilde{v}_i$
    be a lift of $v_i$ to $\ind_K^G M$, and let $L$ be the finite rank $A$-representation of $K$ generated by the
    $\tilde{v}_i$.  Then we have a map $\gamma: \ind_{K}^G L \rarrow \ind_{K}^G M$, and the kernel of the (surjective)
    composition $\ind_{K}^G M \rarrow V_2 \rarrow V_3$ is the sum of the image of $\alpha$ and the image of $\gamma$,
    and so is finitely generated.
  \item Choose a presentation $\ind_{K}^G N \rarrow \ind_{K}^G M \xrightarrow{\alpha} V_3 \rarrow 0$.  We may replace
    $M$ by its image in $V_3$, so that we have $M \subset V_3$ and $\ind_K^GM \rarrow V_3$ is the natural map.  Let
    $m_1, \ldots, m_r$ generate $M \subset V_3$ as an $A$-module, and for each $i$ let $\tilde{m}_i \in V_2$ be a lift
    of $m_i$.  Let $\tilde{M}$ be the $A[K]$-span of the $\tilde{m}_i$ in $V_2$.  Then there is a surjective map of $K$
    representations $\tilde{M} \rarrow M$, and we let $L$ be the kernel.  There is also a map
    $\beta : \ind_K^G \tilde{M} \rarrow V_2$ giving a commuting diagram with exact rows and columns:
    \[  \begin{CD}@. @. \ind_K^G N @ . \\
        @. @. @VVV @ .\\
        \ind_K^G L @>>>\ind_K^G \tilde{M} @>>> \ind_K^G M @>>> 0 \\
        @. @V{\beta}VV @VV{\alpha}V @. \\
        @.  V_2 @>>> V_3 @>>> 0.
      \end{CD}\]
    Repeating the same argument, we may replace $N$ by an $A[K]$-submodule of $\ind_K^GM$ and find a $K$-submodule
    $\tilde{N} \subset \ind_K^G \tilde{M}$, together with a surjection $\tilde{N} \rarrow N$ of $A[K]$-modules, such that
    \[  \begin{CD}\ind_K^G \tilde{N} @>>>\ind_K^G N \\
         @VVV @VVV  \\
        \ind_K^G \tilde{M} @>>> \ind_K^GM .
      \end{CD}\] commutes and has surjective horizontal maps.  The kernel of $\ind_K^G \tilde{M} \rarrow V_3$ is the image
    of $\ind_K^G(\tilde{N} \oplus L)$. Write
    $\gamma$ for the restriction of $\beta$ to $\ind_K^G (\tilde{N} \oplus L)$.  We obtain a commutative diagram
    \[  \begin{CD}@. \ind_K^G (\tilde{N} \oplus L) @>>> \ind_K^G \tilde{M} @>>> V_3 @>>> 0 \\
        @. @V{\gamma}VV @V{\beta}VV @| @. \\
        0@>>> V_1 @>>> V_2 @>>> V_3 @>>> 0
      \end{CD}\] with exact rows, from which we see that $\cok(\gamma) \cong \cok(\beta)$.  As $V_2$ is finitely
    generated, so is $\cok(\beta)$ and hence also $\cok(\gamma)$.  Since $\im(\gamma)$ is also finitely generated, we
    see that $V_1$ is finitely generated by Lemma~\ref{lem:fg-basic}.
  \item Choose surjections $\alpha : \ind_K^G M \rarrow V_1$ and $\beta: \ind_K^G N \rarrow V_3$.  As before, we may
    assume that $N \subset V_3$.  Let $n_1, \ldots, n_r$ generate $N$ as an $A$-module, lift them to
    $\tilde{n}_i \in V_2$, and let $\tilde{N}$ be the $A[K]$-module generated by the $\tilde{n}_i$.  Let $\gamma$ be the
    resulting map $\ind_K^G \tilde{N} \rarrow V_2$.  If we let $L = \ker(\tilde{N}\rarrow N)$, then $\gamma$ restricts
    to a map $\gamma' : \ind_K^GL \rarrow V_1$.  We obtain a commuting diagram
    \[  \begin{CD}0 @>>> \ind_K^G (M \oplus L) @>>> \ind_K^G(M \oplus \tilde{N}) @>>> \ind_K^G N @>>> 0 \\
        @. @V{\alpha + \gamma'}VV @V{\alpha + \gamma}VV @VV{\beta}V @. \\
        0@>>>  V_1 @>>> V_2 @>>> V_3 @>>> 0
      \end{CD}\] with exact rows and surjective vertical maps.  By the snake lemma there is a short exact sequence
    \[0 \rarrow\ker(\alpha+ \gamma') \rarrow \ker(\alpha + \gamma) \rarrow \ker(\beta) \rarrow 0.\] Since the outer two
    terms are finitely generated by (2), so is the inner term (by Lemma~\ref{lem:fg-basic}).  Thus $V_2$ is finitely
    presented, as required. \qedhere
  \end{enumerate}
\end{proof}

\begin{lemma} \label{lem:finite-index} Suppose that $G' \subset G$ is a finite index open subgroup.  Then a $K$-finite
  $A$-representation $V$ of $G$ is finitely generated/presented if and only if its restriction to $G'$ is.
\end{lemma}

\begin{proof}
  \begin{enumerate}
  \item If $V$ is finitely generated as a representation of $G'$ then it certainly is as a representation of $G$.
    Conversely, for any compact open subgroup $K$ of $G$ and any finite rank $A$-representation $W$ of $K$, we have the
    Mackey formula
    \[\mathrm{res}^G_{G'}\ind_K^G W \cong \bigoplus_{g \in G' \backslash G /K} \ind_{gKg^{-1} \cap G'}^{G'} W^{g}.\]
    So $\ind_K^G W$ is finitely generated --- in fact finitely presented --- as a representation of $G'$.  It follows that
    any finitely generated representation of $G$ is finitely generated as a representation of $G'$.

  \item We showed in (1) that $\ind_K^GW $ is finitely presented as a representation of $G'$ for any finite rank
    $A$-representation $W$ of a compact open subgroup $K$.  It follows from Lemma~\ref{lem:fp-basic}~(1) that any
    $K$-finite finitely presented representation of $G$ is finitely presented as a representation of $G'$.

    Conversely, suppose that $V$ is finitely presented as a representation of $G'$.  By the first part, it is finitely
    generated as a representation of $G$, so that there is a surjection $\ind_K^GW \rarrow V$.  Since the first term is
    finitely generated as a representation of $G'$ by part~(1), by Lemma~\ref{lem:fp-basic}~(2) the kernel is finitely
    generated as a representation of $G'$, and hence also as a representation of $G$.  Therefore $V$ is finitely
    presented as a representation of $G$ by Lemma~\ref{lem:fp-basic}~(1). \qedhere
  \end{enumerate}
\end{proof}
\subsection{$Z$-finiteness}
\label{sec:centre}

Suppose that $G$ is a locally profinite group with centre $Z$.  We say that \emph{Hypothesis Z is satisfied} if, for
some (equivalently, any) compact open subgroup $K$ of $G$, $Z/K\cap Z$ is finitely generated.  Recall from
\cite{MR2667882} the definitions of $Z$-finite and locally $Z$-finite representations: a representation is $Z$-finite if
the action of $A[Z]$ on $V$ factors through a quotient $A[Z]/I$ that is a finitely generated $A$-module.  It is locally
$Z$-finite if the $A[Z]$-module spanned by any $v \in V$ is a finitely generated $A$-module.  By \cite{MR2667882}
Lemma~2.3.3, a representation of $G$, finitely generated as an $A[G]$-module, is $Z$-finite if and only if it is locally
$Z$-finite.

\begin{lemma} \label{lem:Z-fp} Let $V$ be a locally $Z$-finite, $K$-finite, $A$-representation of $G$.
  \begin{enumerate}
  \item The representation $V$ is finitely generated if and only if there is a surjection
    \[ \ind_{KZ}^G W \rarrow V \rarrow 0\]
    for some compact open subgroup $K$ of $G$ and finite rank $A$-representation $W$ of $KZ$.
  \item If the representation $V$ is finitely presented then there is an exact sequence
    \[ \ind_{K_1Z}^G W_1 \rarrow \ind_{K_2Z}^G W_2 \rarrow V \rarrow 0\] for some compact open subgroup $K$ of $G$ and
    finite rank $A$-representations $W_1$ and $W_2$ of $K_1Z$ and $K_2Z$.  If Hypothesis Z is satisfied, the converse
    holds.
  \end{enumerate}
\end{lemma}

\begin{proof}
  \begin{enumerate}
  \item The backwards implication is clear.  For the forwards implication, let $W$ be the $A[KZ]$-span of a finite set
    of generators of $V$.  It is finite-rank since $V$ is $K$-finite and locally $Z$-finite.  We therefore get a
    surjection $\ind_{KZ}^G W \rarrow V \rarrow 0$ as required.
  \item Suppose that $V$ is finitely presented.  Then there is a surjection $\ind_{KZ}^G W_2 \rarrow V \rarrow 0$, by
    the first part.  The kernel is finitely generated by Lemma~\ref{lem:fp-basic}~(2), and $\ind_{KZ}^G W_2$ is
    $Z$-finite.  Applying the first part again, we get an exact sequence
    $\ind_{KZ}^G W_1 \rarrow \ind_{KZ}^G W_2 \rarrow V \rarrow 0$ as required.

    For the other direction, it is enough to show that (under Hypothesis Z) $\ind_{KZ}^G W_2$ is finitely presented for
    any representation $W_2$ of $KZ$ on a finitely generated $A$-module.  If $U$ is the kernel of the natural map
    $\ind_{K}^{KZ}W_2 \rarrow W_2$ then there is a short exact sequence
    \[0 \rarrow \ind_{KZ}^G U \rarrow \ind_{K}^G W_2 \rarrow \ind_{KZ}^G W_2 \rarrow 0.\] We have to show that $U$ is
    finitely generated as a $KZ$-representation.  This follows from Hypothesis Z, since this implies that $A[KZ/K]$ is a
    noetherian ring. \qedhere
  \end{enumerate}
\end{proof}

Now suppose that $H$ is an open subgroup of $G$ such that $HZ$ has finite index in $G$ and $Z \cap H$ is compact.

\begin{proposition}\label{prop:Z-fp} Let $V$ be a locally $Z$-finite, $K$-finite, $A$-representation of $G$.
  \begin{enumerate}
  \item The representation $V$ of $G$ is finitely generated if and only if its
    restriction to $H$ is finitely generated.
  \item If the representation $V$ of $G$ is finitely presented then its restriction to $H$ is finitely presented.  If
    Hypothesis Z holds, then the converse is true.
  \end{enumerate}

\end{proposition}
\begin{proof}
By Lemma~\ref{lem:finite-index} we may assume that $G = HZ$.  Let $V$ be a locally $Z$-finite $K$-finite
    representation of $G$.
  \begin{enumerate}
  \item If $V$ is finitely generated as an representation of $H$, it certainly is as a representation of $G$.
    Conversely, suppose that $V$ is finitely generated as an representation of $G$.  If $W \subset V$ is a finitely
    generated $A$-module that generates $V$ as an representation of $G$, then the $Z$-span $ZW$ is a finitely generated
    $A$-module that generates $V$ as a representation of $H$.  So $V$ is a finitely generated representation of $H$ as
    required.

  \item Suppose that $V$ is finitely presented as a representation of $G$.  By Lemma~\ref{lem:Z-fp}~(2) and
    Lemma~\ref{lem:fp-basic}~(1), it suffices to show that $\ind_{KZ}^G W$ is a finitely presented
    representation of $H$ for $K \subset H$.  This follows from the identity of representations of $H$
    \[\ind_{KZ}^{HZ} W = \ind_{K(Z\cap H)}^H W\]
    and the assumption that $Z \cap H$ is compact.

    Finally, suppose that $V$ is finitely presented as an representation of $H$ and that Hypothesis Z holds.  Then $V$ is
    finitely generated as a representation of $G$, so by Lemma~\ref{lem:Z-fp}~(1) there is a surjection
    $\ind_{KZ}^G W \rarrow V$.  By~(1) and Lemma~\ref{lem:fp-basic}~(2) the kernel of this map is a finitely
    generated representation of $H$.  By~(1) again, it is a finitely generated representation of $G$, and so by
    Lemma~\ref{lem:Z-fp}~(1) we have an exact sequence
    \[\ind_{KZ}^G U \rarrow \ind_{KZ}^G W \rarrow V \rarrow 0.\] As $G$ satisfies Hypothesis Z, by the converse
    direction of Lemma~\ref{lem:Z-fp}~(2), $V$ is a finitely
    presented representation of $G$.\qedhere
  \end{enumerate}
\end{proof}

\section{Completed group rings.}
\label{sec:compl-group}

If $K$ is a profinite group, let
\[A[[K]] = \varprojlim_{J \normal K \text{open}} A[K/J]\]
be the completed group ring, a compact topological $A$-algebra.

\begin{lemma}\label{lem:grpring} Suppose that $M$ is a finite rank $A$-representation of $K$.  Then there is a unique
  $A[[K]]$-module structure on $M$ extending the $A[K]$-module structure.
\end{lemma}
\begin{proof}
  For each $n$, the action of $A[K]$ on $M\otimes_A A/\mf_A^n$ factors through $A[K/J_n]$ for some open subgroup $J_n
  \subset K$ and so extends uniquely to an action of $A[[K]]$.  Since $M$ is finitely generated as an $A$-module, $M =
  \varprojlim M \otimes_A A/\mf_A^n$ and the lemma follows.
\end{proof}

Kohlhaase \cite{MR3682662} has extended the notion of completed group ring beyond the compact case. Let $G$ be a locally
profinite group.

\begin{proposition}[Kohlhaase] \label{prop:kohlhaase}
  If $K \subset G$ is a compact open subgroup, then there is a unique $A$-algebra structure on
  \[\AG = A[G] \otimes_{A[K]} A[[K]]\]
  such that the natural maps $A[G] \rarrow \AG$ and $A[[K]] \rarrow \AG$ are $A$-algebra homomorphisms. This
  $A$-algebra is independent of the choice of $K$ up to canonical isomorphism.
\end{proposition}
\begin{proof}
  This is shown in section~1 of \cite{MR3682662} when $A$ is a field --- where what we call $\AG$ is denoted
  $\Lambda(G)$ --- but the proof works verbatim for general rings $A$.  We recall the construction for the reader's
  convenience.  Firstly, if $K'$ is an open subgroup of $K$, then the natural map of $(A[G], A[[K']])$-bimodules
  \[\rho_{K,K'} : A[G] \otimes_{A[K']} A[[K']] \rarrow A[G] \otimes_{A[K]} A[[K]]\]
  is an isomorphism.\footnote{In \cite{MR3682662} this is stated for $K'$ normal in $K$, but it is true for any $K'$ and
    moreover this is necessary for the construction of the ring structure.}  If $K'' \subset K'$ then we have
  $\rho_{K, K''} = \rho_{K',K''} \circ \rho_{K, K'}$, and so we may construct the direct limit
  \[\AG = \varinjlim_{K}\left( A[G] \otimes_{A[K]} A[[K]]\right)\]
  which is (canonically) isomorphic to any one of its terms.  Now, if $g \in G$ then there is an isomorphism of direct
  systems \[ \cdot g : A[G] \otimes_{A[K]} A[[K]] \rarrow A[G] \otimes_{A[g^{-1}Kg]} A[[g^{-1}Kg]]\] taking
  $h \otimes \kappa$ to $hg \otimes g^{-1}\kappa g$.  This defines a right action of $G$ on $\AG$ by left $A[G]$-module
  isomorphisms, which suffices to define the required ring structure on $\AG$.  Precisely, if
  $h \otimes \kappa \in A[G] \otimes_{A[K]} A[[K]]$ and $h' \otimes \kappa' \in A[G] \otimes_{A[K']} A[[K']]$ are
  representatives of elements of $\AG$, we may assume that $K \subset h^{-1}K'h$ and define
  \[(h' \otimes \kappa')(h \otimes \kappa) = h'h \otimes h^{-1}\kappa' h \kappa \in A[G] \otimes_{A[h^{-1}K'h]} A[[h^{-1}K'h]].\qedhere\]
\end{proof}

For later use, we record a flatness result:

\begin{lemma} \label{lem:flat} 
  The $A$-algebra $\AG$ is flat as a right $A[[K]]$-module for any compact open subgroup $K$ of $G$.
\end{lemma}

\begin{proof}
  As in \cite{MR3682662}, $\AG = A[G] \otimes_{A[K]} A[[K]] \cong \bigoplus_{h \in G/K} A[[K]]$ as right $A[[K]]$-modules, so that $\AG$ is
  even a free right $A[[K]]$-module.  
\end{proof}

\begin{remark}
  In the same way we could put an $A$-algebra structure on $A[[K]] \otimes_{A[K]}A[G]$ (for any compact open subgroup $K$) and
  the $A$-module map
  $A[[K]] \otimes_{A[K]}A[G]\rarrow \AG$ defined by
  \[ \kappa \otimes h \mapsto h \otimes h^{-1}\kappa h \in A[G] \otimes_{A[h^{-1}Kh]}A[[h^{-1}Kh]]\] is an isomorphism
  of $A$-algebras.  Thus Lemma~\ref{lem:flat} holds with `right' replaced by `left'.
\end{remark}

\begin{lemma} \label{lem:action} Suppose that $V$ is a $K$-finite $A$-representation of $G$.  Then there
  is a unique $\AG$-module structure on $V$ extending the $A[G]$-module structure.
\end{lemma}
\begin{proof}
  Since $V$ is $K$-finite, for any compact open subgroup $K$ the action of $A[K]$ extends uniquely to an action of
  $A[[K]]$ by Lemma~\ref{lem:grpring}.  By the unicity, we have that, for any $h \in G$ and $\kappa \in A[[K]]$, the two
  actions of $h^{-1}\kappa h$ defined on the one hand by the actions of $G$ and $A[[K]]$, and on the other hand by the
  action of $A[[h^{-1}Kh]]$, agree.  From the formula for multiplication in $\AG$ given in
  Proposition~\ref{prop:kohlhaase}, it follows that we can define an action of $\AG$ on $V$ by fixing $K$ and setting
  \[ (h \otimes \kappa)(v) = h(\kappa(v))\] for any $h \in A[G]$ and $\kappa \in A[[K]]$, which is clearly the unique
  action extending those of $A[G]$ and $A[[K]]$.
\end{proof}

\begin{lemma} \label{lem:fg} Suppose that $V$ is a $K$-finite $A$-representation of $G$. Then $V$ is finitely generated
  if and only if it is finitely generated as an $\AG$-module.
\end{lemma}

\begin{proof}
  Suppose that $V$ is finitely generated.  By Lemma~\ref{lem:fg-grp-ring}, $V$ is finitely generated as an
  $A[G]$-module, and hence as a $\AG$-module. 

  Conversely, let $V$ be a $K$-finite $A$-representation of $G$ that is finitely generated as a $\AG$-module.  Then, if
  $v_1, \ldots, v_r$ generate $V$ and if $M$ is their $A[K]$-span, then $M$ is also preserved by $A[[K]]$ and so
  \[V = \AG\cdot M = (A[G] \otimes_{A[K]} A[[K]]) M = A[G]M.\]
  Therefore $V$ is finitely generated, as required.
\end{proof}
The key technical reason for us to introduce the ring $\AG$ is that it \emph{is} true that a finitely presented
$K$-finite $A$-representation of $G$ is a finitely presented $\AG$-module --- see Remark~\ref{rmk:not-true}.  The
starting point is the following result of Lazard (see \cite{MR2667882} Theorem~2.1.1).

\begin{theorem} \label{thm:lazard} If $G$ is a $p$-adic analytic group, then $A[[K]]$ is noetherian for every compact
  open subgroup $K$ of $G$. \qed
\end{theorem}

\begin{proposition} \label{prop:fp} Suppose that $G$ is a $p$-adic analytic group.  Let $V$ be a $K$-finite
  $A$-representation of $G$.  Then $V$ is finitely presented if and only if it is finitely presented as an $\AG$-module.
\end{proposition}

\begin{proof} 
  The backwards implication follows from Lemma~\ref{lem:fg}.  Suppose that $V$ is finitely presented
  as an $\AG$-module.  Then by
  Lemma~\ref{lem:fg} there is a surjection
  \[\alpha : \ind_K^G W \rarrow V \rarrow 0\] for some finite rank $A$-representation $W$ of a
  compact open subgroup $K \subset G$.  The kernel of $\alpha$ is a $K$-finite representation of $G$ that is finitely
  generated as an $\AG$-module, by \cite[Tag 0519]{stacks-project}~(5).\footnote{Strictly speaking, \cite[Tag
    0519]{stacks-project} is only stated for modules over commutative rings.  However, it is still true, with an
    identical proof, in the non-commutative case.}  Therefore it is finitely generated as an $A$-representation of $G$,
  by Lemma~\ref{lem:fg}.

  Suppose now that $V$ is finitely presented.  Then by Lemma~\ref{lem:fp-basic}~(2), there is a compact open subgroup
  $K$, a finite rank $A$-representation $M$ of $K$, and a surjection $\ind_K^G M \rarrow V \rarrow 0$ with finitely
  generated kernel.

  By \cite[Tag 0519]{stacks-project}~(4) and Lemma~\ref{lem:fg}, it is enough to show that $\ind_K^G(M)$ is a finitely
  presented $\AG$-module for the $\AG$-module structure provided by Lemma~\ref{lem:action}.  We may think of this
  instead as the tensor product
  \[\ind_K^G(M)\cong A[G] \otimes_{A[K]} M\]
  via the isomorphism sending an element $f : G \rarrow M$ of $\ind_K^G(M)$ to $\sum_{g \in G/K} g \otimes f(g^{-1})$.
  By Lemma~\ref{lem:grpring} the action of $A[K]$ on $M$ extends uniquely to one of $A[[K]]$ and we have isomorphisms
  \[\AG \otimes_{A[[K]]} M = A[G] \otimes_{A[K]} \otimes A[[K]] \otimes_{A[[K]]} M = A[G] \otimes_{A[K]} M\]
  of $A[G]$-modules, and hence of $\AG$-modules (by Lemma~\ref{lem:action}).  

  Since $A[[K]]$ is noetherian by Theorem~\ref{thm:lazard}, the finitely generated $A[[K]]$-module $M$ is finitely
  presented; let $A[[K]]^m \rarrow A[[K]]^n \rarrow M \rarrow 0$ be a presentation.  Applying $\AG \otimes_{A[[K]]} -$,
  we obtain an exact sequence
  \[\AG^n \rarrow \AG^m \rarrow \AG \otimes_{A[[K]]}M = A[G] \otimes_{A[K]} M \rarrow 0\]
  so that $\ind_K^G M = A[G] \otimes_{A[K]} M$ is a finitely presented $\AG$-module, as required.\qedhere
\end{proof}

\section{Amalgamations and coherence}
\label{sec:amalg}

Let $K_1, K_2$ and $I$ be profinite groups equipped with inclusions $f_i : I \into K_i$ of $I$
as a common open subgroup of $K_1$ and $K_2$.  Then there are maps $f_i:A[[I]] \rarrow A[[K_i]]$ of
topological augmented $A$-algebras.

Let $H = K_1 *_I K_2$ be the amalgamation of $K_1$ and $K_2$ along $I$.  By \cite{MR0476875}, Th\'{e}or\`{e}me 1, the
natural map $I \rarrow H$ is injective.  The following proposition shows that $H$ is naturally a locally profinite
topological group:

\begin{proposition}
  With the colimit topology,\footnote{The coarsest topology on $H$ such that for every topological group $G$ equipped
    with continuous maps $K_i \rarrow G$ agreeing on $I$, there is a continuous map $H \rarrow G$ extending these.} $H$
  is a locally profinite group with a basis of open neighbourhoods of the identity being given by open neighbourhoods of
  $I$.
\end{proposition}

\begin{proof}
  Let $H$ and $H'$ respectively denote $H$ with the colimit topology and the topology for which translates of open
  subgroups of $I$ are a basis of open sets.  Let $i : H \rarrow H'$ and $j : H' \rarrow H$ be the identity maps; we
  have to show that they are both continuous.  But $i$ is continuous by the universal property of $H$, and $j$ is
  continuous because the map $I \rarrow H$ is continuous.
\end{proof}

We now consider the amalgamated product of rings, $A[[K_1]] *_{A[[I]]} A[[K_2]]$.  Note first that
$A[K_1]*_{A[I]}A[K_2]$ is simply the group ring of $H$ over $A$.  This is because the functor $G \mapsto A[G]$ from
groups to $A$-algebras is a left-adjoint, and so commutes with the colimit $*$.

In general, we have $A$-algebra maps $A[[K_1]] \rarrow \AH$ and $A[[K_2]] \rarrow \AH$ which agree on $A[[I]]$, and so (by the
universal property) an $A$-algebra map $\alpha:A[[K_1]] *_{A[[I]]} A[[K_2]] \rarrow \AH$.
\begin{proposition} The map
  \[\alpha:A[[K_1]] *_{A[[I]]} A[[K_2]] \rarrow \AH\]
  is an isomorphism of $A$-algebras.
\end{proposition}
\begin{proof}
  Let $R = A[[K_1]] *_{A[[I]]} A[[K_2]]$.
  
  The composite map \[A[H] = A[K_1] *_{A[I]}A[K_2] \rarrow R \rarrow \AH\] is easily seen to be the natural map
  $A[H] \rarrow \AH$.  It follows that the image of $R$ in $\AH$ contains $A[H]$ and $A[[K_i]]$ and so in fact is all of
  $\AH$, whence $\alpha$ is surjective.

  Moreover, from the universal property of $\otimes$, for each $i$ we have a map of $(A[H], A[[K_i]])$-bimodules
  \[A[H] \otimes_{A[K_i]} A[[K_i]] \rarrow R\] and these define the \emph{same} map $\beta : \AH \rarrow R$.  Since this
  is a map of right $A[[K_1]]$- and $A[[K_2]]$-modules, we see that $\beta \circ \alpha$ is the identity --- it is
  enough to check that it takes $1$ to $1$.  Therefore $\alpha$ is injective and so an isomorphism.
\end{proof}

\subsection{Coherence.}
\label{sec:coherence-ah}

Recall that a ring $R$ is (left) coherent if any of the following equivalent definitions hold:

\begin{enumerate}
\item every finitely generated left ideal of $R$ is finitely presented;
\item if $f: M \rarrow N$ is a map of finitely presented left $R$-modules, then $\ker(f)$ is
  finitely presented;
\item the category of finitely presented left $R$-modules is an abelian subcategory of the category of left $R$-modules.
\end{enumerate}

\begin{proposition} \label{prop:aberg} If the rings $A[[K_i]]$ are coherent and $A[[I]]$ is noetherian,
  then $\AH$ is coherent.
\end{proposition}
\begin{proof}
  This follows immediately from \cite{MR680365} Theorem~12; the hypotheses of that theorem are satisfied, by
  Lemma~\ref{lem:flat}. For the convenience of the reader, we summarise the argument of \cite{MR680365} in the case of
  interest to us.  It uses the characterisation --- due to Chase \cite{MR0120260} --- of left coherent rings as those
  for which arbitrary products of right flat modules are flat.  Let $R$, $S$ and $T$ be rings such that $S$ and $T$ are
  $R$-algebras, and $Q = S*_R T$ is flat as a right $R$, $S$ or $T$-module; we will take $R = A[[I]]$ and
  $S = A[[K_1]]$, $T = A[[K_2]]$. Then there is a Mayer--Vietoris sequence for $\Tor^Q$ in terms of $\Tor^S$, $\Tor^R$
  and $\Tor^T$.  If $R$ is left noetherian and $S$ and $T$ are left coherent, then take a set $(F_i)_{i \in I}$ of right
  flat $Q$-modules and compare the Mayer--Vietoris sequence for $\Tor(\prod F_i, M)$ with the product of those for
  $\Tor(F_i, M)$, for an arbitrary left $Q$-module $M$ This gives $\Tor_i^Q(\prod F_i, M) = 0$ for $i > 1$. Since
  $S$ and $T$ are left coherent and that, as $R$ is left noetherian and the $F_i$ are right flat $R$-modules,
  $(\prod F_i) \otimes_R M \rarrow \prod(F_i \otimes_R M)$ is injective by \cite{MR680365} Lemma~6.  It follows that
  $\Tor_1^Q(\prod F_i, M)$ also vanishes, so that $\Tor_i^Q(\prod F_i, M) = 0$ for all $i > 0$ as required.
\end{proof}

Combining with Theorem~\ref{thm:lazard} we get:

\begin{corollary} \label{cor:coherent} Suppose that $H$ is a $p$-adic analytic group that is an amalgamated product of
  two compact open subgroups.  Then $\AH$ is coherent. \qed
\end{corollary}

\begin{theorem} \label{thm:rep-coherent} Suppose that $H$ is a $p$-adic analytic group that is an amalgamated product of two
  compact open subgroups.  Then the category of finitely presented $K$-finite $A$-representations of $H$
  is an abelian subcategory of the category of $A$-representations of $H$.
\end{theorem}

\begin{proof}
  It suffices to show that the kernel or cokernel of a map of finitely presented $K$-finite $A$-representations of $H$
  is also a finitely presented $K$-finite $A$-representation.  This is straightforward for cokernels, and
  does not require the ring $\AH$.  For kernels, suppose that $f : V \rarrow W$ is a map of
  finitely presented $K$-finite $A$-representations of $H$.  Then $\ker(f)$ is a $K$-finite
  $A$-representation of $H$, and by Proposition~\ref{prop:fp} and Corollary~\ref{cor:coherent} it is
  finitely presented as a left $\AH$-module.  By Proposition~\ref{prop:fp} again, it is a finitely presented
  $A$-representation of $H$. 
\end{proof}

\section{Applications.}
\label{sec:applications}

Let $F$ be a local field of characteristic 0 with ring of integers $\Oc_F$ and residue field $k$ of
characteristic $p$, and let $D$ be a division algebra over $F$ with ring of integers $\Oc_D$.
Choose a uniformiser $\pi$ of $D$. Let $G= GL_2(D)$ and let $G' = SL_2(D)$ be the subgroup of
elements of reduced norm $1$.    Let $K_1 = GL_2(\Oc_D)$ and let
$K'_1 = SL_2(\Oc_D) = K' \cap SL_2(D)$.  Let $\alpha = \twomat{1}{0}{0}{\pi} \in G$, and
let $K_2 = \alpha K_1 \alpha^{-1}$ and $K'_2 = K_2 \cap G'$.  Let
\[I = K_1 \cap K_2 = \left\lbrace\twomat{a}{b}{c}{d} \in K_1 : c \equiv 0 \bmod \pi\right\rbrace\]
and $I' = I \cap G' = K'_1 \cap K'_2$.

\begin{theorem}
\label{thm:sl2}
  The category of finitely presented $K$-finite $A$-representations of $G'$ is an abelian subcategory of
  $\Cc^{\Kfin}_A(G')$.
\end{theorem}
\begin{proof}
  By a theorem of Ihara (Serre \cite{MR0476875} Chapter~II Corollary~1) we know that
  $G' = K'_1 *_{I'} K'_2$.  The theorem follows from Theorem~\ref{thm:rep-coherent}.
\end{proof}

\begin{corollary}
\label{cor:gl2}
  The category of finitely presented, $K$-finite, (locally) $Z$-finite $A$-representations of $G$ is an abelian subcategory of $\Cc^{\Kfin}_A(G)$.
\end{corollary}

\begin{proof}
  Let $G^0$ be the subgroup of $G$ of elements whose reduced norm is in $\Oc_F^\times$ and let $Z$ be the centre of $G$.
  Then $ZG^0$ has finite index in $G$, $Z \cap G^0$ is compact, and $Z/Z\cap K$ is finitely generated for any compact
  open subgroup $K$ of $G$.  Let $f: V_1 \rarrow V_2$ be a map of $K$-finite $Z$-finite finitely presented representations
  of $G$.  By Proposition~\ref{prop:Z-fp} they are finitely presented representations of $G^0$.  By \cite{MR0476875}
  Chapter~II Theorem~3, $G^0 = K_1 *_I K_2$, and so Theorem~\ref{thm:rep-coherent} the kernel $\ker(f)$ is finitely
  presented as a representation of $G^0$.  By Proposition~\ref{prop:Z-fp} again, it is a finitely presented
  representation of $G$.
\end{proof}
\bibliography{references.bib}{}
\bibliographystyle{amsalpha}
\end{document}